\documentclass[12pt,a4paper]{article}

\usepackage{graphicx}				
\usepackage{color}				
\usepackage{amsmath}				
\usepackage{amsthm}
\usepackage{theoremref}
\usepackage{amsfonts}
\usepackage{here}
\usepackage[utf8x]{inputenc}

\newtheorem{theorem}{Theorem}
\newtheorem{open}{Open Problem}

\author{\small Paul F. X. Müller and Katharina Riegler}
\date{\small \today}
\title{\large Radial Variation of Bloch functions on the unit ball of $\mathbb{R}^d$}

\begin{document}

\maketitle

\begin{abstract}
In \cite{Jo} Anderson's conjecture was proven by comparing values of Bloch functions with the variation of the function. We extend that result on Bloch functions from two to arbitrary dimension and prove that
\begin{equation*}
\int\limits_{[0,x]} |\nabla b(\zeta)|e^{b(\zeta)}d|\zeta| < \infty.
\end{equation*}
In the second part of the paper, we show that the area integral 
\begin{equation*}
\int\limits_{B^d} |\nabla u(w)|p(w,\theta)dA(w)
\end{equation*}
for positive harmonic functions $u$ is bounded by the value $cu(0)$ for at least one $\theta$. The integral is also transferred to simply connected domains and interpreted from the point of view of stochastics. Several emerging open problems are presented. 
\end{abstract}

\section{Introduction}
In \cite{Jo} it was proven that for every Bloch function $b$ on the unit disc there is a point $x$ on the unit circle such that 
\begin{equation*}
\int\limits_{[0,x]} |\nabla b(\zeta)|e^{b(\zeta)}d|\zeta| < \infty.
\end{equation*}
This fact was used in \cite{Jo} to show Anderson's conjecture for conformal maps on the unit disc. The proof used the result of  Bourgain (\cite{Bourgain}) that for every positive harmonic function there is a direction of bounded radial variation, Pommerenke's theorem on the existence of a dense set of rays along which a Bloch function on the unit disc remains bounded (see (\cite[Proposition 4.6.]{Pom}) and conformal mappings between starlike Lipschitz domains and the unit disc. The result of Bourgain (\cite{Bourgain}) was extended to half-spaces by Michael O'Neill in \cite{Neill} and even to higher dimensional Lipschitz domains by Mozolyako and Havin in \cite{MozHav}. The extension to Lipschitz domains in higher dimensions by Mozolyako and Havin was the starting point of our work in this paper. The main result of the present paper is the following \thref{Bloch}, extending the result of \cite{Jo} to the unit ball of $\mathbb{R}^d$. We refer to \cite{Nicolau} and \cite{Neill} where this problem was posed in writing. 
\begin{theorem}
\thlabel{Bloch}
Let $b$ be a Bloch function on the unit ball $B$ of $\mathbb{R}^d$. Then there is a point $x$ on the unit sphere such that
\begin{equation*}
\int\limits_{[0,x]} |\nabla b(\zeta)|e^{b(\zeta)}d|\zeta| < \infty.
\end{equation*}
\end{theorem}

\paragraph*{Remarks} 
The proof of \thref{Bloch} will show as well that there is an $x$ such that for $y\in [0,x]$ and $1-|y|$ small enough we have 
\begin{equation*}
b(y)\leq b(0)-c\int\limits_{[0,y]}|\nabla b(\zeta)|d|\zeta|.
\end{equation*}
As $-b$ is also a Bloch function we get a point $\tilde{x}$ on the unit sphere such that for $y\in [0,\tilde{x}]$ and $1-|y|$ small enough we have
\begin{equation*}
b(y)\geq b(0)+c\int\limits_{[0,y]}|\nabla b(\zeta)|d|\zeta|.
\end{equation*}
This can also be written as 
\begin{equation*}
\liminf\limits_{y\rightarrow 1} \frac{b(y)-b(0)}{\int\limits_{[0,y]}|\nabla b(\zeta)|d|\zeta|}>0.
\end{equation*}

Comparing our proof of \thref{Bloch} with the one in \cite{Jo} we observe the following.
\begin{enumerate}
\item The use of conformal maps onto Lipschitz domains is replaced by the result of Mozolyako and Havin. 
\item In \cite{Jo} the straight line segments obtained by Pommerenke's theorem are used to suitably cut planar domains. While Pommerenke's result on the radial behaviour of Bloch functions is still valid in $\mathbb{R}^d$ for $d > 2$ (see Nicolau \cite{Nicolau}) the corresponding line segments cannot be used to seperate domains in $\mathbb{R}^d$ for $d>2$ (Note that 1=d-1 iff d=2).
\item Case 3 in our construction of points within a cone (see Figure \ref{third case}) provides a substitute for the use of Pommerenke's result in \cite{Jo}.
\end{enumerate}

Analysing the proof by Mozolyako and Havin, we are able to show the following theorem, which was conjectured by Peter W. Jones in 2003 during conversations with the first named author.  
\begin{theorem}
\thlabel{integral}
Let $u$ be a positive harmonic function on the unit ball of $\mathbb{R}^d$ and $p$ the Poisson kernel. Then there is a point $\theta$ on the unit sphere $S$ such that
\begin{equation*}
\int\limits_{B^d} |\nabla u(w)|p(w,\theta)dA(w)<c u(0)
\end{equation*}
where $c=c(B^d)$ is a constant only depending on the dimension. 
\end{theorem}

Note that the area integral is obviously larger than the line integrals defining the usual radial variation. \thref{integral} is proved in section \ref{Proof of Th 2}.

In section \ref{last section} we further investigate the significance of the integral in \thref{integral}. Utilizing its conformal invariance we are able to transfer it to arbitrary simply connected domains and also provide its stochastic interpretation using Brownian motion. We complement our work in section \ref{last section} by presenting several connected open problems.

\section{Preliminaries}
We use the notation $B$ for the unit ball of $\mathbb{R}^d$, $S$ for its boundary the unit sphere, $\mathbb{D}$ for the unit disc of $\mathbb{C}$, $B_r$ or $B(r)$ for the ball with center $0$ and radius $r$. The Euclidean distance between two points or a set and a point will be denoted by $d(\cdot,\cdot)$ and the diameter of a set $A$ with $\textrm{diam}(A)$.
For domains $E$ their boundary is denoted by $\partial E$, the inward unit vector of a point $x$ of $\partial E$, if it is well-defined, by $N(x)$.

\paragraph*{Hyperbolic distance/metric:}
On $\mathbb{D}$ the hyperbolic length of a smooth curve $\gamma$ is given by 
\begin{equation*}
2\int\limits_{\gamma}\frac{d|z|}{1-|z|^2}.
\end{equation*}
The hyperbolic distance between two points $z_0$ and $z_1$ is the infimum of the hyperbolic lengths of all piecewise smooth curves in $\mathbb{D}$ with endpoints $z_0$ and $z_1$. It is invariant under conformal self-mappings of the disc. The geodesics in this metric are circles orthogonal to $\{|z|=1\}$. The distance from $0$ to an arbitrary point $z_0\in \mathbb{D}$ is given by $\log\left(\frac{1+|z_0|}{1-|z_0|}\right)$. See \cite[Section 4.6]{Pom}.

\paragraph*{Bloch functions:}
A function $b$ on the unit ball $B$ of $\mathbb{R}^d$ is called a Bloch function if it is harmonic and the semi-norm $\|b\|:=\sup\limits_{z\in B} |\nabla b(z)|(1-|z|)$ is finite. Bloch functions are Lipschitz with respect to the hyperbolic metric. This means by definition that there is a constant $L$ such that for all $z,w\in B^d$
\begin{equation*}
|b(z)-b(w)|\leq L d_h(z,w)
\end{equation*} 
where $d_h(z,w)$ is the hyperbolic distance between $z$ and $w$.
See \cite[Section 4.2]{Pom}.

\paragraph*{Poisson kernel:}
The Poisson kernel on the unit ball $p:\mathbb{B^d}\times S \rightarrow\mathbb{R}$ is given by $p(z,\zeta):=\frac{1-|z|^2}{\omega_{d-1}|\zeta-z|^d}$ for $|z|<1$, $|\zeta|=1$ and $\omega_{d-1}$ the surface area of the unit sphere.
Especially in section \ref{Proof of Th 2} we use $p_r(\theta,\zeta)$ instead of $p(r\theta,\zeta)$, generating a family of kernels $(p_r)_{r\in [0,1)}$. Analogously, we write $u_r(\zeta)$ for $u(r\zeta)$ for functions $u$ on the unit ball.  

\paragraph*{Green's function:}
A Green's function for a domain $\Omega$ is a function $g:\Omega\times\Omega\rightarrow (-\infty,\infty]$ such that for each $w\in\Omega$
\begin{enumerate}
\item $g(\cdot,w)$ is harmonic on $\Omega\setminus \{w\}$ and bounded outside each neighbourhood of $w$
\item $g(w,w)=\infty$ and as $z\rightarrow w$
\begin{equation*}
g(z,w)=
\begin{cases}
\log|z|+O(1)&w=\infty\\
-\log|z-w|+O(1)&w\neq \infty
\end{cases}
\end{equation*}
\item $g(z,w)\rightarrow 0$ as $z\rightarrow\zeta$ and $\zeta \in \partial\Omega$.
\end{enumerate}
See \cite[Section 4.4]{Ransford}.

\paragraph*{Harnack's inequality:}
We will use Harnack's inequality to compare values of positive harmonic functions and to get a bound for their gradients. 
\begin{theorem}
Let $h$ be a positive harmonic function on the disc $B(w,\rho)$. Then for $r<\rho$ and $0\leq t<2\pi$
\begin{equation*}
\frac{\rho-r}{\rho+r}\leq h(w+re^{it})\leq \frac{\rho+r}{\rho-r}.
\end{equation*}
\end{theorem}
See \cite[Theorem 1.3.1]{Ransford}.

\paragraph*{Harmonic measure and harmonic majorant:}
We use the notation $w^{z_0}(F,E)$ for the harmonic measure with pole $z_0$ of $F\subset \partial E$.
A harmonic majorant of a function on a given domain is a harmonic function which is pointwise larger or equal to the function.

\paragraph*{Martin boundary:}
To define the Martin boundary of a domain we consider $M(x,y):= \frac{g(x,y)}{g(x_0,y)}$ where $x_0$ is a fixed point in the domain. The function $x\mapsto M(x,y)$ is continuous for $y\in\Omega\setminus\{x\}$. We now use the theorem of Constantinescu-Cornea (see \cite[Theorem 7.2]{Bass} or \cite[Theorem 12.1]{Helms}) to get a compact set $\Omega^*$, unique up to homeomorphisms such that
\begin{enumerate}
\item $\Omega$ is a dense subset of $\Omega^*$,
\item for each $y\in\Omega$ the function $x\mapsto M(x,y)$ has a continuous extension to $\Omega^*$ and
\item the extended functions separate points of $\Omega^*\setminus\Omega$. 
\end{enumerate}
The set $\Omega^*\setminus\Omega$ is the Martin boundary of $\Omega$ and denoted by $\partial_M \Omega$. The extensions of $M$ are called Martin kernels and denoted by $k^{\Omega}$. Martin kernels provide the following fundamental representation theorem for positive harmonic functions. 
\begin{theorem}
For every positive harmonic function $h$ on $\Omega$ there is a measure $\nu$ concentrated on $\partial_M\Omega$ such that 
\begin{equation*}
h(x)=\int k^\Omega(x,y)d\nu(y).
\end{equation*}
\end{theorem}
See \cite[Section II.7]{Bass} or \cite[Chapter 12]{Helms}.

\paragraph*{Riemann Mapping Theorem}
In the last section we will use the Riemann mapping theorem to transfer \thref{integral} to arbitrary simply connected domains. 
\begin{theorem}
Let $\Omega$ be a simply connected proper subdomain of $\mathbb{C}$ and $w_0\in \Omega$. Then there is a conformal map $\gamma:\Omega\rightarrow \mathbb{D}$ with $\gamma(w_0)=0$.
\end{theorem}
See \cite[Theorem 4.4.11]{Ransford}.

\paragraph*{Prime ends:}
In a simply connected and bounded domain $\Omega\subset\mathbb{C}$ a crosscut $C$ of $\Omega$ is an open Jordan arc in $\Omega$ such that $\bar{C}=C\cup \{a,b\}$ with $a,b\in\partial\Omega$. A sequence $(C_n)$ of crosscuts of $\Omega$ is called a null-chain if $\bar{C}_n\cap\bar{C}_{n+1}=\emptyset$, $\textrm{diam}(C_n)\rightarrow 0$ and $C_{n+1}$ and $C_0$ are in different components of $\Omega \setminus C_n$. The component of $\Omega\setminus C_n$ not containing $C_0$ is called $V_n$. Two null-chains $(C_n)$ and $(C'_n)$ are equivalent if for every sufficiently large $m$ there exists $n$ such that $V_n\subset V'_m$ and $V'_n\subset V_m$. The equivalence classes of null-chains are called prime ends of $\Omega$. The set of prime ends is denoted by $P(\Omega)$. 

The ordinary topology on $\Omega$ is extended in the following way. For a subdomain $A\subset \Omega$, $\mathcal{E}_A$ is the set of prime ends that contain a null-chain whose crosscuts all lie in $A$. We define the new topology by adding the set $A\cup \mathcal{E}_A$ as a neighbourhood of each point in $A$ and each prime end in $\mathcal{E}_A$. In this extended topology $\Omega$ is dense in $\Omega\cup P(\Omega)$, $\Omega\cup P(\Omega)$ is a compact space and therefore called prime end compactification of $\Omega$. Now the following holds true.
\begin{theorem}
If $\gamma:\Omega\rightarrow \mathbb{D}$ is a conformal homeomorphism, it can be extended to a homeomorphism $\hat{\gamma}$ between the prime end compactification of $\Omega$ and $\bar{\mathbb{D}}$.
\end{theorem}
The prime ends statisfy the conditions (1-3) of the Martin boundary. As the Martin boundary is unique up to homeomorphisms, for simply connected domains prime ends and Martin boundary coincide in this sense.
See \cite[Section 2.4]{Pom} or \cite[Chapter 9]{CoLo}.

\paragraph*{Coarea formula:}
For an open set $\Omega\subset \mathbb{R}^d$ a real-valued Lipschitz function $a$ and an $\mathcal{L}^1$ function $b$ we have 
\begin{equation*}
\int\limits_{\Omega} |\nabla a(x)| b(x)dx=\int\limits_{\mathbb{R}} \int\limits_{a^{-1}(t)}g(x) d\sigma_t(x) dt 
\end{equation*}
where $\sigma_t$ is the $n-1$ dimensional Hausdorff-measure on the preimage of $t$ under the function $a$.
See \cite[Section 3.4.2]{Evans}.

\paragraph*{Brownian motion and local time:}
We call $(X_t)_{t\in \mathbb{R}^+}$ Brownian motion started at $0$ if it is a stochastic process on a probability space $(\Omega,\mathcal{F}, \mathbb{P})$ such that
\begin{enumerate}
\item $X_0=0$ $\mathbb{P}$-almost surely,
\item the increments $X_{t_1}-X_{t_0}$, $X_{t_2}-X_{t_1}$, $X_{t_m}-X_{t_{m-1}}$ are independent for $0\leq t_1<t_2<...<t_m$,
\item $X_t-X_s\sim \mathcal{N}(0,t-s)$ for $0\leq s\leq t$, in other words
$\mathbb{P}(X_t-X_s\in A)=\frac{1}{\sqrt{2\pi(t-s)}}\int\limits_A e^{-\frac{1}{2}\left(\frac{t}{\sqrt{t-s}}\right)^2}dt$ and
\item the paths are almost surely continuous. 
\end{enumerate}
If $X^{(1)},X^{(2)},\cdots,X^{(d)}$ are independent Brownian motions starting in $0$, the stochastic process given by $Z=(X^{(1)},X^{(2)},\cdots,X^{(d)})_t$ is called a $d$-dimensional Brownian motion. 
The limit $L^x_t:=\lim\limits_{\epsilon\rightarrow 0} \frac{1}{2\epsilon}\int 1_{\{x-\epsilon<X_s<x+\epsilon\}} ds$ exists and is called local time of Brownian motion. 
See\cite[Section I.2,I.6]{Bass}.

\section{Proof of Theorem 1}

The proof of \thref{Bloch} is based on a result of Mozolyako and Havin (\cite{MozHav}). 
We use their result to find a point in a given subset of the boundary of a Lipschitz domain such that the variation along the normal to the boundary is bounded. 

The theorem of Mozolyako and Havin (\cite{MozHav}) will be used as stated in \thref{Moz}. 
Note however that the original statement in \cite{MozHav} involves $\mathcal{C}^2$ domains. However, it turns out - and is known - that their argument may be modified so as to work for Lipschitz domains. 
\begin{theorem}
\thlabel{Moz}
Let $u$ be a positive harmonic function on a Lipschitz domain $O$ with starcenter $z_0$ and boundary $D$. Let $N(p)$ be any direction at $p$ pointing "well-inside" the domain and $r$ be a positive function on $D$ such that $[p,p+r(p)N(p)]\subset O$ for all $p\in D$. Then for all surface balls $E\subset D$ with $\omega^{z_0}(E,O)\geq c$ there is a $p_0\in E$ and a harmonic majorant $H$ of the gradient such that 
\begin{equation*}
\int\limits_{0}\limits^{r(p_0)}  H(p+yN(p)) dy < c_1u(z_0)
\end{equation*}
where the constant $c_1$ only depends on the Lipschitz constant of the domain, the constant $c$ and the Harnack distance between $z_0$ and $p_0+r(p_0)N(p_0)$. 
\end{theorem}

\paragraph*{Proof of Theorem 1}
The proof consists of 4 parts. 
\begin{enumerate}
\item First of all we construct Lipschitz domains on which we can use \thref{Moz}.
\item In the second part we look for points within a cone such that the variation along the segments connecting the points is suitably bounded.
\item Part 3 is dedicated to shifting the points constructed in part 2 onto one radius of the unit ball.
\item In the last part we collect the information of the previous parts to prove the theorem.
\end{enumerate}

\paragraph*{Construction of Lipschitz domains:}

We want to construct a Lipschitz domain $W(z_0)$ for an arbitrary point $z_0\in B$  on which $b$ is bounded from below. In the following $M\in \mathbb{N}$ is a large enough constant only depending on $\sup\limits_{z\in B} |\nabla b(z)|(1-|z|)$. The Lipschitz domain should satisfy the following conditions. 

\begin{enumerate}
\item The function $b-b(z_0)+M$ is positive on the whole domain.
\item The domain $W(z_0)$ has starcenter $z_0$ and the Lipschitz constant is independent of $z_0$. 
\end{enumerate}

\begin{itemize}
\item For $z \in B$ with $|z|>\frac{15}{16}$ we use the following notations:
\begin{align*}
I(z)&:=\{\zeta\in S: |z-\zeta|\leq 8(1-|z|)\}\\
r(z)&:=2|z|-1\\
T(z)&:=\{w\in\mathbb{B}: |w|=r(z), |w-z|\leq 2(1-|z|)\}.
\end{align*}
The domain $V(z)$ is the convex hull of $T(z)$ and $I(z)$, intersected with $B\setminus B_{r(z)}$ (see Figure \ref{V(z)}).

\begin{figure}[H]
	\centering
\includegraphics[width=0.9\textwidth, angle=0]{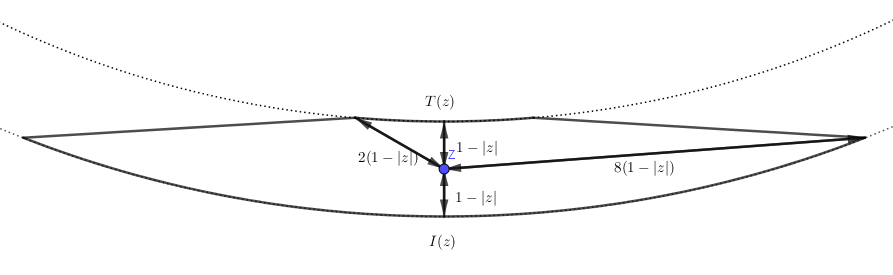}
	\caption{The domain $V(z)$}
\label{V(z)}	
\end{figure}

\item 
We consider the Whitney decomposition $(C_n)_{n\in \mathbb{N}}$ of $B$, so we have a decomposition of $B$ into disjoint cubes with
\begin{equation*}
\textrm{diam}(C_n)\leq d(C_n,S)\leq 4\, \textrm{diam}(C_n)
\end{equation*}
for $n\in \mathbb{N}$. Next, we fix $n\in \mathbb{N}$, check if there is a point $z\in C_n$ satisfying
\begin{equation*}
b(z)-b(z_0)\leq -M
\end{equation*}
and set $N_T:=\{n\in \mathbb{N}: \exists z \in C_n: b(z_n)-b(z_0)\leq -M\}$. For $n\in N_T$ select any $z_n\in C_n$ such that $b(z_n)-b(z_0)\leq -M$ holds true. Finally, we define $T:=\{z_n:n\in N_T\}$.

\item 
For all $z\in T$ we set $\tilde{V}(z)=V(z)$ if $\textrm{dist}(T(z),\partial V(z_0))>\frac{1}{2}(1-|z|)$. If the condition does not hold true, we  set $\tilde{T}(z):=\{w\in\mathbb{B}: |w|=r(z), |w-z|\leq 3(1-|z|)\}$ and $\tilde{V}(z)=\textrm{conv}(\tilde{T}(z)\cup I(z))\cap (B\setminus B_{r(z)})$. The boundary of $\tilde{V}(z)$ consists of three disjoint sets: the top $\tilde{T}(z)$, the bottom $\partial \tilde{V}(z)\cap S$ and $\tilde{L}(z)$.

\item 
The Lipschitz domain we were looking for is now given by 
\begin{equation*}
W(z_0):=V(z_0)\setminus \bigcup\limits_{z_n\in T}\tilde{V}(z_n).
\end{equation*}
Once again the boundary of $W(z_0)$ consists of different parts. First, we have $T(z_0)$ and $\partial W(z_0)\cap L(z_0)$. The new bottom is divided into three subsets $\partial W(z_0)\cap S$, the tops $\partial W(z_0)\cap\bigcup\limits_{z\in T} T(z)$ and $\partial W(z_0)\cap\bigcup\limits_{z\in T} L(z)$.

On the tops, so for $z\in\partial W(z_0)\cap\bigcup\limits_{z\in T} T(z)$ we have
\begin{equation*}
-2M < b(z)-b(z_0) < -\frac{M}{2}.
\end{equation*}
\end{itemize}

\paragraph*{Construction of points within a cone:}

In the following we construct a sequence of points such that the integral of a harmonic majorant of the gradient  along the segments connecting two consecutive points (i.e. the variation) is bounded. 

In the first step we look at $x_0=0$ and the corresponding domain $W(0)$. Now, by \thref{Moz}, we find a point $x_1$ in $\partial W(0)$ such that the variation of $b$ along the line segment connecting $x_1$ and $y_1:=x_1+N(x_1)$ is bounded by $c_1 M$. As $y_1$ has small hyperbolic distance from $x_0$, the variation along the interval $[y_1,x_0]$ is bounded because of Harnack's inequality.

Now we assume that points $x_0,...,x_n$ and $y_1,...,y_n$ are already chosen. We consider the domain $W(x_n)$ and scale it by a homothetic transformation such that the diameter is $100$. As $b-b(x_n)+M$ is a positive harmonic function on $W(x_n)$ we can apply \thref{Moz}. We get a point $x$ on the lower part of the boundary of $W(x_n)$ such that the variation along the line segment connecting $x$ and $y_{n+1}:=x+N(x)$ is bounded by $c_1 M$. We distinguish three cases.
\begin{enumerate}
\item If $x$ is in the unit sphere, we take $x_{n+1}:=x$ and stop the construction.

\item If $x$ is on a top, so in the  set $\partial W(x_n)\cap\bigcup\limits_{z\in T} T(z)$, we define $x_{n+1}:=x$.

\item If $x$ is  in $\partial W(x_n)\cap\bigcup\limits_{z\in T} L(z)$, we take the intersection point $y$ of the normal to the boundary at $x$ and the sphere with radius $r(z_n)$, which is the radius of the corresponding top, and set $x_{n+1}:=y$ (see Figure \ref{third case}).

\begin{figure}[H]
	\centering
  \includegraphics[width=0.7\textwidth, angle=0]{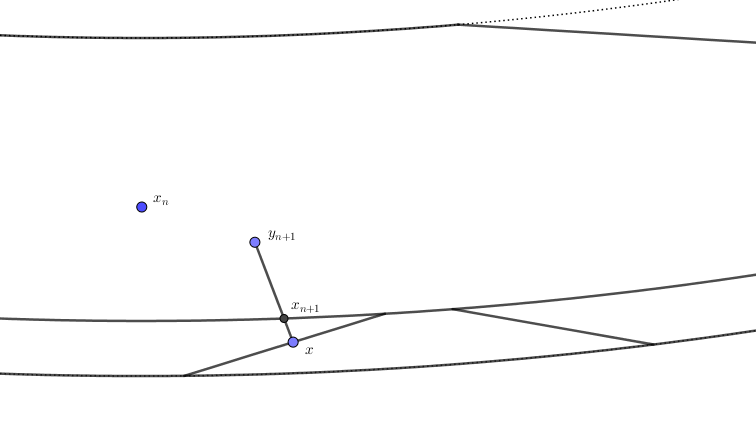}
	\caption {The case $x\in \partial W(x_n)\cap\bigcup\limits_{z\in T} L(z)$}
	\label{third case}
\end{figure}
\end{enumerate}

We have constructed two sequences of points $(x_n)_{n\in\mathbb{N}\cup \{0\}}$ and $(y_n)_{n\in\mathbb{N}}$, both converging to a point on the boundary $x$. They are contained in the cone with apex $x$ and opening angle $\pi /2$ and satisfy the following conditions:

\begin{enumerate}
\item $\int\limits_{[y_{n+1},x_{n+1}]}H_{n+1}(\zeta)d|\zeta|\leq c_1 M$,
\item $\int\limits_{[x_{n},y_{n+1}]}H_{n+1}(\zeta)d|\zeta|\leq c_1 M$,
\item $b(x_n)-b(0)\leq -n\frac{M}{2}$ and
\item for $\zeta \in [x_n,y_{n+1}]\cup[y_{n+1},x_{n+1}]$ we have $b(\zeta)-b(0)\leq -n\frac{M}{2}+2c_1M$
\end{enumerate}

where $H_{n+1}$ is a harmonic majorant of $|\nabla b|$ on the domain $W(x_n)$ (see Figure \ref{fig3}).

\begin{figure}[H]
	\centering
  \includegraphics[width=0.5\textwidth, angle=0]{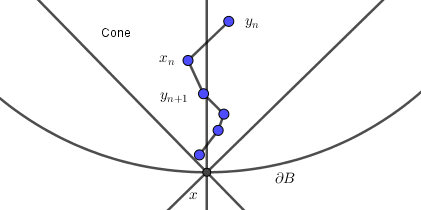}
	\caption {Cone, radius and some points}
	\label{fig3}
\end{figure}

\paragraph*{Shifting of points and corresponding segments:}

We want to shift the points $x_n$ and $y_n$ and the corresponding segments to the radius connecting $0$ and the limit $x$ of our point sequences.

Shifting $[x_n,y_{n+1}]$:
As in the construction for the segment $[x_n,y_{n+1}]$, we can use Harnacks inequaltity to bound the integral of the harmonic function along the segment $[a_n,b_{n+1}]=[|x_n| x, |y_{n+1}|x]$.

Shifting $[y_{n+1},x_{n+1}]$:
We shift $y_{n+1}$ to $b_{n+1}=|y_{n+1}|x$ and $x_{n+1}$ to $a_{n+1}=|x_{n+1}| x$ and distinguish the following two cases. 

Case 1: $|x_{n+1}|=1$. In this case we have already reached the unit sphere and $x_{n+1}=x$ so we already have a suitable bound for
\begin{equation*}
\int \limits_{[b_{n+1},a_{n+1}]} |\nabla b(\zeta)|d|\zeta|
\end{equation*} 
as $y_{n+1}=b_{n+1}$ and $x_{n+1}=a_{n+1}$. 

Case 2: If $|x_{n+1}|\neq 1$, we note that the distance of $x_{n+1}$ to the radius is bounded by $c|x_{n+1}|$ because the point is in the cone. Next, we take $s:=[y_{n+1},x_{n+1}]\cap B(0,1-\frac{3}{2} (1-|x_{n+1}|))$. The distance of $s$ to the boundary of $W(x_n)$ is at least $\frac{1}{2}(1-|x_{n+1}|)$ and therefore comparable to the distance from the radius to which we want to shift our segment. We can use Harnack's inequality once again and get a bound of
\begin{equation*}
\int\limits_{[(1-\frac{3}{2} (1-|x_{n+1}|))x,b_{n+1}]} |\nabla b(\zeta)|d|\zeta|
\end{equation*}
where $b_{n+1}:=|y_{n+1}|x$ (see Figure \ref{fig4}).

\begin{figure}[H]
	\centering
  \includegraphics[width=0.8\textwidth, angle=0]{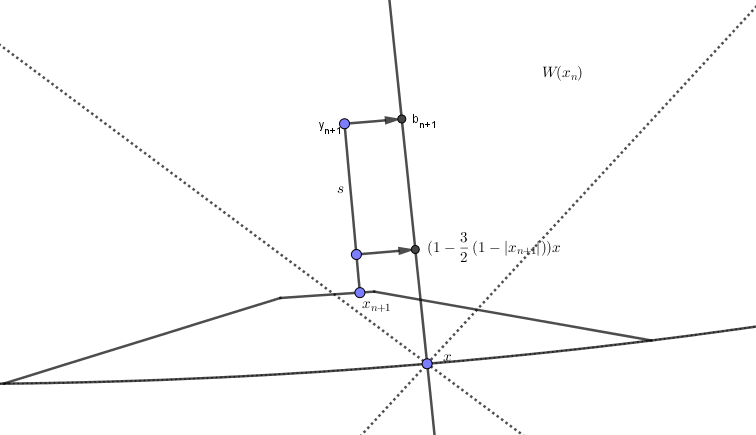}
	\caption {Case 2}
	\label{fig4}
\end{figure}

The segment $[a_{n+1},(1-\frac{3}{2} (1-|x_{n+1}|))x]$ with $a_{n+1}:=|x_{n+1}|x$ is left. Here we use the next harmonic majorant $H_{n+2}$ on the domain $W(x_{n+1})$. There we can easily bound 
\begin{equation*}
\int \limits_{[|x_{n+1}|x,(1-\frac{3}{2} (1-|x_{n+1}|))x]}H_{n+2}(\zeta)d|\zeta|
\end{equation*}
by $c_1M$.

\paragraph*{Collecting the information:}
We obtained a point $x$ on the unit sphere and sequences of points $(a_n)$ and $(b_n)$ on $[0,x]$ both converging to $x$ that satisfy:

\begin{enumerate}
\item $|a_n|<|b_{n+1}|<|a_{n+1}|$
\item $\int\limits_{[b_{n+1},a_{n+1}]}|\nabla b(\zeta)|d|\zeta|\leq c_1 M$
\item $\int\limits_{[a_{n},b_{n+1}]}|\nabla b(\zeta)|d|\zeta|\leq  c_1 M$
\item $b(a_n)-b(0)\leq -n\frac{M}{2}$
\item for $\zeta \in [a_n,b_{n+1}]\cup[b_{n+1},a_{n+1}]$ we have $b(\zeta)-b(0)\leq -n\frac{M}{2}+2c_1M$.
\end{enumerate}

To finish the proof we can write this in a more compact form as

\begin{align*}
&\int\limits_{[0,x]} |\nabla b(\zeta)|e^{b(\zeta)}d|\zeta|=\\
&=\sum \limits_{n=0}\limits^{\infty} \int\limits_{[a_{n+1},b_{n+1}]}|\nabla b(\zeta)|e^{b(\zeta)}d|\zeta|+\int\limits_{[b_{n+1},a_n]} |\nabla b(\zeta)|e^{b(\zeta)}d|\zeta|\leq\\
&\leq 2 e^{b(0)+2c_1 M}\sum\limits_{n=0}\limits^{\infty} e^{-n\frac{M}{2}} c_1 M <c(c_1,M,b(0)).
\end{align*}

\section{Proof of Theorem 2}\label{Proof of Th 2}

In the previous section we used the result of Mozolyako and Havin (\cite{MozHav}) to prove \thref{Bloch}. In this section we present a result of the analysis of the proof given in \cite{MozHav} to prove \thref{integral}.

Let $u$ be a positive harmonic function on $B^d$. Following the notation of Mozolyako and Havin in \cite{MozHav} we define $\sigma(x)$ as the normalized gradient of $u$ at the point $x$. By differentiating the Poisson kernel with respect to the first variable in direction $\sigma(x)$, we get a new kernel  $c_r(\zeta,\theta):=\frac{\partial p}{\partial \sigma(r^2\zeta)}(r\zeta,\theta)$.
We define the kernel $b_r(\zeta, \theta):=\int p_r(\zeta,x)c_r(x,\theta)dx$ and $B_r$ denotes the inegral operator $B_r(f)=\int \limits_S b_r(\zeta,\theta)f(\theta)d\theta$.

Mozolyako and Havin (\cite{MozHav}) proved that for every positive harmonic function $u$ on $B^{d}$ there is a probability measure $\nu$ such that
\begin{equation*}
J=\int\limits_{S}\int\limits_{0}\limits^{1} B_r(u_r)(\zeta)drd\nu(\zeta)\leq cu(0). 
\end{equation*}

Let $V(u)$ be the set of points in $S$ such that the radial variation of $u$ is bounded.
In \cite{MozHav} they showed that for every point $p\in S$ and $\rho > 0$ the Hausdorff dimension of $V(u)\cap B^d(p,\rho)$ is $n-1$.
We use the same fact to prove \thref{integral}.  

\begin{proof}
First we rewrite the integral in the theorem as follows
\begin{equation*}
\int\limits_{B^d} |\nabla u(w)|p(w,\theta)dA(w)=\int\limits_S\int\limits_S\int\limits_0\limits^1 c_{\sqrt{r}}(\zeta,\alpha)u_{\sqrt{r}}(\alpha)p_r(\zeta,\theta)r\,drd\zeta d\alpha.
\end{equation*}
Using the symmetry of $p_r$ and because $p_r\leq 2p_{\sqrt{r}}$ pointwise, we obtain 

\begin{equation*}
\int\limits_0\limits^1 2 B_{\sqrt{r}}(u_{\sqrt{r}})(\theta)dr
\end{equation*}
as an upper bound for our integral and by a simple substitution this is equal to  
\begin{equation*}
4\int\limits_0\limits^1 B_s(u_s)(\theta)ds.
\end{equation*}
This integral is bounded by $c u(0)$ for at least one $\theta$ as $\nu$ is a probability measure. 

\end{proof}

By similar computation
\begin{equation*}
I=\int\limits_S \int\limits_S\int\limits_0\limits^1 P_r(|\nabla u_r(\zeta)|)p_r(\zeta,\theta)dr d\zeta d\nu(\theta),
\end{equation*} 
where $P_r$ is the operator with integral kernel $p_r$, is also bounded.

\section{Discussion and Related Open Problems} \label{last section}

In the following discussion we exploit the remarkable flexibility of the integral
\begin{equation*}
\int\limits_{B^d} |\nabla u(w)|p(w,\theta)dA(w)
\end{equation*}
in \thref{integral}, in particular its conformal invariance and its connection to Brownian martingales. 

\paragraph*{First Application of the Coarea Formula}

By applying the coarea formula to the integral in \thref{integral} we get that 
\begin{equation*}
\int \limits_{0}\limits^{\infty}\int\limits_{u^{-1}(c)}p(w,\theta)dH_{n-1}(w)dc <c u(0)
\end{equation*}
where $H_{n-1}$ is the $n-1$ dimensional Hausdorff measure and $u^{-1}(c)$ is the preimage of the value $c$. 
For a harmonic function $v$ with $|v|<A$ we get
\begin{equation*}
\int \limits_{-A}\limits^{A}\int\limits_{v^{-1}(c)}p(w,\theta)dH_{n-1}(w)dc 
\end{equation*} 
and as there is a $\theta$ such that the expression is bounded, we know that $\int\limits_{v^{-1}(c)}p(w,\theta)dH_{n-1}$ is in $\mathcal{L}^1([-A,A])$. Therefore we can ask the following question:

\begin{open}
For which $p>1$ exists a $\theta$ such that the integrand $\int\limits_{v^{-1}(c)}p(w,\theta)dH_{n-1}$ is in $\mathcal{L}^p([-A,A])$?
\end{open}

\paragraph*{Simply Connected Domains:}

In order to get from an arbitrary simply connected domain to the unit disc we use the Riemann mapping theorem.

We make use of two concepts regarding the boundary of simply connected domains. The first one is the Martin boundary of a domain (cf \cite{Bass} or \cite{Helms}) and the second one is the concept of prime ends introduced by Carathéodory (cf \cite{Pom} and \cite{CoLo}).
For simply connected domains, prime ends and the Martin boundary coincide up to homeomorphisms. 

We can now take advantage of the conformal invariance and transfer our result of \thref{integral} to this setting.

\begin{theorem}
\thlabel{simply connected}
Let $\Omega$ be a simply connected bounded domain, $v$ a positive harmonic function on $\Omega$, $k^\Omega$ the Martin kernel and $g^\Omega(\cdot,w_0)$ the Green's function with singularity in $w_0\in\Omega$. Then there is a prime end $\zeta$ such that 
\begin{equation*}
\int \limits_{\Omega} |\nabla v(w)|k^{\Omega}(w,\zeta)|\nabla g^{\Omega}(w,w_0)|dA(w)<c
\end{equation*} 
where $c$ is a constant only depending on the value $v(w_0)$ and the domain $\Omega$.
\end{theorem}

\begin{proof}
Let $\gamma:\Omega\rightarrow\mathbb{D}$ be a Riemann map with $\gamma(w_0)=0$. The function $u:=v\circ f$, where $f$ is the inverse of $\gamma$, is a positive harmonic function on $\mathbb{D}$. So we know by \thref{integral} that there is a $\theta\in\mathbb{T}$ with
\begin{equation*}
\int \limits_{\mathbb{D}} |\nabla(v\circ f)(z)|p(z,\theta)dA(z)<c v(w_0).
\end{equation*}
By substitution we get that 
\begin{equation*}
\int \limits_{\Omega} |\nabla v(w)|p(\gamma(w),\theta)||\gamma'(w)|dA(w)<c v(w_0).
\end{equation*}
Next we use that $p(\gamma(w),\theta)$ is the Martin kernel $k^{\Omega}(w,\zeta)$ where $\zeta$  is the prime end of $\Omega$ such that the extension $\hat{\gamma}$ of $\gamma$ to the prime ends of $\Omega$ satisfies $\hat{\gamma}(\zeta)=\theta$ and get 

\begin{equation*}
\int \limits_{\Omega} |\nabla v(w)|k^{\Omega}(w,\zeta)|\gamma'(w)|dA(w)<c v(w_0).
\end{equation*}

By calculation we know $|\gamma'(w)|=|\nabla g^{\Omega}(w,w_0)||\gamma(w)|$ where $g^{\Omega}$ is the Green's function of $\Omega$. Therefore we obtain 
\begin{equation*}
\int \limits_{\Omega} |\nabla v(w)|k^{\Omega}(w,\zeta)|\nabla g^{\Omega}(w,w_0)||\gamma(w)|dA(w)<c v(w_0).
\end{equation*}
We now use the fact that for bounded simply connected plane domains the integral of the gradient of the Green's function 
\begin{equation*}
\int\limits_\Omega|\nabla g^\Omega(w,z)|dA(z)
\end{equation*}
is bounded by a constant only depending on $\Omega$ (see \cite{Evans}). 
Partitioning the domain into one part where $|\gamma| <\frac{1}{2}$ and one where $|\gamma|\geq \frac{1}{2}$ leads to the following. On the domain where $|\gamma| <\frac{1}{2}$ we know that $k^\Omega$ and $|\nabla v|$ are bounded by constants only depending on the domain $\Omega$ and the value $v(w_0)$. On the second part we only use that $|\gamma|\geq \frac{1}{2}$. Then we obtain
\begin{equation*}
\int \limits_{\Omega} |\nabla v(w)|k^{\Omega}(w,\zeta)|\nabla g^{\Omega}(w,w_0)|dA(w)<c
\end{equation*} 
where the constant $c$ only depends on the value $v(w_0)$ and $\Omega$. 
\end{proof}

The area integral in \thref{simply connected} depends expressly on the Green's function and the Martin kernels of the simply connected domain $\Omega\subset\mathbb{C}$ and not on the Riemann map itself. This allows us to consider it in more general domains, either in multiply connected domains or in higher dimensions. 

\begin{open}
Let $v$ be a positive harmonic function on a not necessarily simply connected domain $\Omega\subset\mathbb{C}$ with Green's function $g^\Omega$.
Is there an element $\zeta$ of the Martin boundary of $\Omega$ such that the integral
\begin{equation*}
\int\limits_{\Omega}|\nabla v(w)|k^\Omega(w,\zeta)|\nabla g^{\Omega}(w,w_0)|dA(w)
\end{equation*}
is bounded? Even for Denjoy domains, that is when $\partial\Omega\subset\mathbb{R}$, this problem seems to be open.
\end{open}

We refer to Jerison and Kenig (\cite{JerKe}) for the concept of nontangentially accessible (NTA) domains and can ask the following.
\begin{open}
Let $v$ be a positive harmonic function on a nontangentially accessible domain $\Omega\subset\mathbb{R}^d$.
Is there an element $\zeta$ of the Martin boundary of $\Omega$ such that the integral
\begin{equation*}
\int\limits_{\Omega}|\nabla v(w)|k^\Omega(w,\zeta)|\nabla g^{\Omega}(w,w_0)|dA(w)
\end{equation*}
is bounded? The problem is even open for Lipschitz domains.
\end{open}

\paragraph*{Second Application of the Coarea Formula}
We continue by transforming the integral 
\begin{equation*}
\int \limits_{\Omega} |\nabla v(w)|k^{\Omega}(w,\zeta)|\nabla g^{\Omega}(w,w_0)||\gamma(w)|dA(w)
\end{equation*}
Using the coarea formula we get that the integral 
\begin{equation*}
\int \limits_{0}\limits^{\infty}\int\limits_{g^{-1}(c)} |\nabla v(w)|k^{\Omega}(w,\zeta)(w)|\gamma (w)|d\sigma_c(w)dc
\end{equation*}
is bounded. In fact, we apply the coarea formula to a partition of $\Omega$  into subdomains where the Green's function is Lipschitz and put the results together. As we know that $|\gamma (w)|=e^{-g(w,w_0)}$, the expression
\begin{equation*}
\int \limits_{0}\limits^{\infty}\int\limits_{g^{-1}(c)} |\nabla v(w)|k^{\Omega}(w,\zeta)(w)d\sigma_c(w)e^{-c}dc,
\end{equation*}
where $\sigma_c$ is the Hausdorff-measure, is bounded. 
As this implies that 
\begin{equation*}
\int\limits_{g^{-1}(c)} |\nabla v(w)|k^{\Omega}(w,\zeta)(w)d\sigma_c(w) 
\end{equation*}
is in $\mathcal{L}^1([0,\infty),e^{-c} dc)$, we can ask the following question.

\begin{open}
Is there a $\zeta$ such that the function $c\mapsto\int\limits_{g^{-1}(c)} |\nabla v(w)|k^{\Omega}(w,\zeta)(w)d\sigma_c(w)$ is in $\mathcal{L}^p([0,\infty),e^{-c} dc)$ for $p>1$?
\end{open}

\paragraph*{Brownian martingales:}

We also want to present the link to stochastic analysis. Therefore, let $Z_t$ be the $d$-dimensional Brownian motion stopped at the unit sphere. We apply the formula
\begin{equation*}
\mathbb{E}\left(\int \limits_0\limits^\tau m(Z_t) dt\right)=\int\limits_{B^d}m(z)g(z,0)dA(z)
\end{equation*}
for bounded continuous functions $m$, the Green's function $g$ and Brownian motion $Z_t$.
Now the integral in \thref{integral} can be written as

\begin{align*}
\int \limits_{B^d}|\nabla u(w)|p(w, \theta)dA(w)& =\mathbb{E}\left(\int\limits_{0}\limits^{\tau}|\nabla u(Z_t)|p(Z_t,\theta)\frac{1}{g(0,Z_t)}dt\right).
\end{align*}
As $u(Z_t)$ can be written as $\int\limits_0\limits^t \nabla u (Z_s)dZ_s$ for $t<\tau$, a new question arises naturally:
\begin{open}
For $F:\Omega \rightarrow \mathbb{R}^+$ with $F=\int\limits_{0}\limits^{\tau}G_t dZ_t+F_0$, where $G_t$ is a martingale, show that there is a $\theta$ such that $\mathbb{E}\left(\int\limits_{0}\limits^{\tau}|G_t|\frac{p(Z_t,\theta)}{g(0,Z_t)}dt\right)<c$.
\end{open}

Using the local time of Brownian motion $L_\tau^w$ and the fact $\mathbb{E}(L_\tau^w)=g(0,w)$ we could also rewrite the integral of \thref{integral} 
\begin{equation*}
\int\limits_{B^d}|\nabla u (w)| p(w,\theta)\frac{1}{g(0,w)}\mathbb{E}(L_{\tau}^w) dA(w)
=\mathbb{E}\left(\int\limits_{B^d}|\nabla u (w)| p(w,\theta)\frac{1}{g(0,w)}L_{\tau}^w dA(w)\right).
\end{equation*}
These expressions are all bounded for at least one $\theta$ because of their equality to the integral in \thref{integral}.

\subsection*{Acknowledgements}
This paper is part of the second named author’s PhD thesis written at the Department of Analysis, Johannes Kepler University Linz. The  research has  been  supported  by  the  Austrian  Science  foundation  (FWF) Pr.Nr P28352-N32.

\textsc{P.F.X. M\"uller, Institute of Analysis, Johannes Kepler University Linz, Altenberger
Strasse 69, A-4040 Linz, Austria}

E-mail address: Paul.Mueller@jku.at

\bigskip

\textsc{K. Riegler, Institute of Analysis, Johannes Kepler University Linz, Altenberger
Strasse 69, A-4040 Linz, Austria}

E-mail address: Katharina.Riegler@jku.at
\end{document}